\documentclass[11pt]{article}

\usepackage{amsmath,amsthm,amssymb}
\usepackage{graphicx}
\usepackage{subfigure}
\usepackage[usenames]{color}
\usepackage{hyperref}

\begingroup\makeatletter\ifx\SetFigFont\undefined%
\gdef\SetFigFont#1#2#3#4#5{%
  \reset@font\fontsize{#1}{#2pt}%
  \fontfamily{#3}\fontseries{#4}\fontshape{#5}%
  \selectfont}%
\fi\endgroup%

\newcommand{\E}{{\mathbb{\,E}}}
\newcommand{\I}{{\mathbb{\,I}}}

\newcommand{\ec}{{\mathcal{E}}}

\newcommand{\hc}{{\mathcal{H}}}

\newcommand{\pc}{{\mathcal{P}}}
\newcommand{\Var} {\operatorname{Var}}
\newcommand{\Cov} {\operatorname{Cov}}
\newcommand{\Prob} {\mathbb {P}}
\newcommand{\prob}[1]{\Prob\left(#1\right)}

\newcommand{\eps} {\varepsilon}

\newcommand{\hh} {\mathbb H}

\newcommand{\e}{\mathrm{e}}

\newcommand{\x}{\mathbf{x}}
\newcommand{\Y}{\mathbf{Y}}
\newcommand{\y}{\mathbf{y}}
\newcommand{\z}{\mathbf{z}}

\newcommand{\hcnd}{\hc^{(k)}(n,d)}

\newtheorem{theorem}{Theorem}

\newtheorem{proposition}[theorem]{Proposition}

\newtheorem*{conjecture*}{Conjecture}
\newtheorem*{proposition*}{Proposition}
\theoremstyle{remark}
\newtheorem{remark}{Remark}
\newtheorem*{remark*}{Remark}

\usepackage{setspace}
  \title{ Approximate counting of regular hypergraphs}

  \author{{\large{Andrzej Dudek}}%
\footnote{\footnotesize {Department of Mathematics, Western Michigan University, Kalamazoo, MI, USA. Research supported in part by Simons Foundation Grant \#244712.}}
\and
{\large{Alan Frieze}}%
\footnote{\footnotesize {Department of Mathematical Sciences, Carnegie Mellon University, Pittsburgh, PA, USA. Research supported in part by NSF Grant CCF2013110.}}
\and
{\large{Andrzej Ruci\'nski}}%
\footnote{\footnotesize {Department of Discrete Mathematics, Adam Mickiewicz University, Pozna\'n, Poland. Research supported by the Polish NSC grant N201 604940 and the NSF grant DMS-1102086. Part of research performed at Emory University, Atlanta.}}\\[1pt]
\and
{\large{Matas \v Sileikis}}%
\footnote{\footnotesize {Department of Mathematics, Uppsala University, Sweden. Research supported by the Polish NSC grant N201 604940. Part of research performed at Adam Mickiewicz University, Pozna\'n.}}
}

\begin{document}
  \maketitle
  \begin{abstract} In this paper we asymptotically count $d$-regular $k$-uniform hypergraphs on $n$ vertices, provided $k$ is fixed and $d=d(n)=o(n^{1/2})$.
In doing so, we extend to hypergraphs a switching technique of McKay and Wormald.
  \end{abstract}
  \section{Introduction}
  We consider \emph{$k$-uniform hypergraphs} (or \emph{$k$-graphs}, for short) on the vertex set $V = [n] := \left\{ 1, \dots, n \right\}$. A $k$-graph $H=(V,E)$ is $d$-regular, if the degree of every vertex $v\in V$, $\deg_H(v) :=\deg(v):= |\left\{ e \in E : v \in e \right\}|$ equals $d$. 
  
  Let $\hcnd$ be the class of all $d$-regular $k$-graphs on $[n]$. Note that each $H\in\hcnd$ has $m:=nd/k$ edges (throughout, we implicitly assume that $k|nd$).
We treat $d$ as a function of $n$, possibly constant.

 A result of McKay \cite{McK} contains an asymptotic formula for the number of $n$-vertex $d$-regular graphs, when $d \le \eps n$ for any constant $\eps < 2/9$.
In this paper we present an asymptotic enumeration of all $d$-regular $k$-graphs on a given set of $n$ vertices, where $k \ge 3$ and $d=d(n)$ is either a constant or does not grow with $n$ too quickly. Let $\kappa=\kappa(k)=1$ for $k\ge4$ and $\kappa(3)= 1/2$.
\begin{theorem}
  \label{thm:enum}
  For every $k \ge3$, $1\le d =o(n^\kappa)$, and 
  \begin{equation*}
    |\hcnd| = \frac{(nd)!}{(nd/k)!(k!)^{nd/k}(d!)^n} \exp \left\{ -\frac12(k-1)(d-1) + O\left( (d/n)^{1/2} + d^2/n \right) \right\}.
  \end{equation*}
\end{theorem}
\noindent  The error term in the exponent tends to zero (thus giving the asymptotics of  $|\hcnd|$) if and only if $d = o(n^{1/2})$. Cf.\ an analogous formula for $k=2$ by McKay \cite{McK}, which gives the asymptotics if and only if $d = o(n^{1/3})$. Recently, 
Blinovsky and Greenhill~\cite{BG} obtained more general results counting sparse uniform hypergraphs with given degrees.

Theorem \ref{thm:enum} extends a result from \cite{CFMR} where Cooper, Frieze, Molloy and Reed proved  that formula for $d$ fixed using the by now standard \emph{configuration model} (see \cite{BC, bela, W} for the graph case). Already for graphs, in \cite{McK}, and later in \cite{MW} and \cite{MW2},
this technique was combined with the idea of \emph{switchings}, a sequence of operations on a graph which eliminate loops and multiple edges, while keeping the degrees unchanged and leading to an \emph{almost} uniform  distribution of the simple graphs obtained as the ultimate outcome (but see Remark \ref{alg} in Section \ref{cr}).

To prove Theorem \ref{thm:enum} we apply these ideas together with a modification from \cite{CFK}, where instead of configurations, permutations were used to generate graphs with a given degree sequence.
To describe this modification, consider a generalization of a $k$-graph in which edges are multisets of vertices rather than just sets.
By a \emph{$k$-multigraph} we mean a pair $H = (V,E)$ where $V$ is a set and $E$ is a multiset of $k$-element multisubsets of $V$. Thus we allow both multiple edges and loops, a \emph{loop} being an edge which contains more than one copy of a vertex. We call an edge \emph{proper} if it is not a loop. We say that a $k$-multigraph is \emph{simple} if it is a $k$-graph, that is, if it contains neither multiple edges nor loops. Henceforth, for brevity of notation, we denote an edge of a $k$-multigraph by $v_1\dots v_k$ rather than $\left\{ v_1, \dots, v_k \right\}$.

Given a sequence $\x\in[n]^{ks}$, $s\in{\mathbb N}$, let $H(\x)$ stand for the  $k$-multigraph with
edge multiset
$E = \{x_{ki+1},\dots,x_{ki+k} : i = 0,\dots, s-1\}$
and let $\lambda (\x)$ be the number of loops in $H(\x)$.

Let $\pc=\pc(n,d) \subset [n]^{nd}$ be the family of all permutations of the sequence 
  \begin{equation*}
    \label{eq:multiset}
    \left( \underbrace{1, \dots, 1}_d, \underbrace{2, \dots, 2}_d, \dots, \underbrace{n, \dots, n}_d \right).
  \end{equation*}
  Note that $|\pc|=(nd)!(d!)^{-n}$. Let $\Y = (Y_1, \dots, Y_{nd})$ be chosen uniformly at random from~$\pc$.

In the next section we sketch a proof of Theorem~\ref{thm:enum} together with some auxiliary results.

  \section{Proof of Theorem \ref{thm:enum}}\label{sec:counting}

  \subsection{Setup} 
 Let $\ec$ be the family of those permutations $\y\in\pc$ for which the $k$-multigraph $H(\y)$ has no multiple edges and contains at most
  $$L := \sqrt{nd}$$ loops, but no loops with less than $k-1$ distinct vertices.
  Let   \begin{equation*}
    \ec_l = \left\{ \y \in \ec : \lambda(\y) = l \right\}, \qquad l = 0, \dots, L.
  \end{equation*}
  Note that
  \begin{equation*}
    \ec_0 = \left\{ \y \in \pc : H(\y) \in \hcnd \right\}
  \end{equation*}
   is precisely the family of those permutations from $\pc$ which represent simple $k$-graphs. In turn, for each $H \in \hcnd$ there are $(nd/k)!(k!)^{nd/k}$ permutations $\y\in\ec_0$ with $H(\y)=H$. Therefore, in order to prove Theorem \ref{thm:enum}, it suffices to show that
  \begin{equation}\label{suff}
    |\pc|/|\ec_0| = \exp \left\{ \frac12(k-1)(d-1) + O(\sqrt{d/n} + d^2/n) \right\}.
  \end{equation}
  Our plan is as follows. First, in Proposition \ref{prop:YE}, we prove that
  \begin{equation}\label{eq:pcec}
    |\pc| \sim \left(1 + O\left( \sqrt{d/n} + d^2/n^{k-2}\right)\right)|\ec|.
  \end{equation}
   Note that for $d = o(n^\kappa)$, the error term in \eqref{eq:pcec} tends to zero and is at most the error term in \eqref{suff}. Thus, it is enough to show \eqref{suff} with $|\ec|$ in place of $|\pc|$, which we do by writing
  \begin{equation}\label{eq:sumprod}
    \frac{|\ec|}{|\ec_0|} = \sum_{l = 0}^L \prod_{i = 1}^l \frac{|\ec_i|}{|\ec_{i-1}|},
  \end{equation}
and estimating the ratio $|\ec_l|/|\ec_{l-1}|$ uniformly for every $1 \le l \le L$.

 In what follows it will be convenient to work directly with permutation $\Y$ rather than with the $k$-multigraph $H(\Y)$ generated by it. Recycling the notation, we still call consecutive $k$-tuples $(Y_{ki+1},\dots,Y_{ki+k})$ of $\Y$ \emph{edges}, \emph{proper edges}, or \emph{loops}, whatever appropriate. E.g., we say that $\Y$ contains \emph{multiple edges}, if $H(\Y)$ contains multiple edges, that is, some two edges of $\Y$ are identical as multisets.
We use the standard notation $(x)_a=x(x-1)\cdots(x-a+1)$.

The following proposition implies \eqref{eq:pcec}, because $ \prob{\Y \in \ec} = {|\ec|}/{|\pc|}$.
  \begin{proposition}
    \label{prop:YE}
    If $k\ge3$, then $ \prob{\Y \in \ec} = 1-O(\sqrt{d/n} + d^2/n^{k-2})$.
    \end{proposition}

\noindent    
A simple proof of Proposition \ref{prop:YE} (details can be found in Appendix~\ref{appendix:a}) is based on the first moment method. In particular, the expected numbers of pairs of multiple edges, loops with less than $k-1$ distinct vertices, and all loops are, respectively, $O(d^2/n^{k-2})$, $O(d/n)$, and $\E\lambda(\Y)\sim\frac {k-1} 2(d-1).$ The last formula implies that $\Prob(\lambda(\Y)>L) \le \frac{\E\lambda(\Y)}L = O(\sqrt{d/n}).$

\begin{figure}[h]
\centering
        \subfigure[]
{%
\label{fig:a}
\begin{picture}(0,0)%
\includegraphics{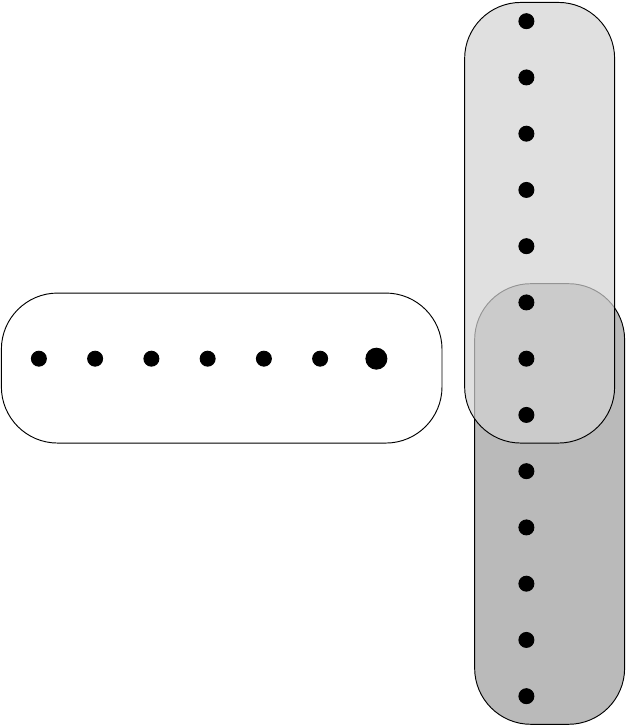}%
\end{picture}%
\setlength{\unitlength}{2368sp}%
\begin{picture}(5127,5799)(5389,-15598)
\put(9751,-12736){\makebox(0,0)[lb]{\smash{{\SetFigFont{7}{8.4}{\rmdefault}{\mddefault}{\updefault}{\color[rgb]{0,0,0}$w_2$}%
}}}}
\put(9751,-11836){\makebox(0,0)[lb]{\smash{{\SetFigFont{7}{8.4}{\rmdefault}{\mddefault}{\updefault}{\color[rgb]{0,0,0}$y_1$}%
}}}}
\put(9751,-11386){\makebox(0,0)[lb]{\smash{{\SetFigFont{7}{8.4}{\rmdefault}{\mddefault}{\updefault}{\color[rgb]{0,0,0}$y_2$}%
}}}}
\put(9751,-12286){\makebox(0,0)[lb]{\smash{{\SetFigFont{7}{8.4}{\rmdefault}{\mddefault}{\updefault}{\color[rgb]{0,0,0}$w_s$}%
}}}}
\put(9751,-10936){\makebox(0,0)[lb]{\smash{{\SetFigFont{7}{8.4}{\rmdefault}{\mddefault}{\updefault}{\color[rgb]{0,0,0}$y_3$}%
}}}}
\put(9751,-10486){\makebox(0,0)[lb]{\smash{{\SetFigFont{7}{8.4}{\rmdefault}{\mddefault}{\updefault}{\color[rgb]{0,0,0}$y_4$}%
}}}}
\put(9751,-13186){\makebox(0,0)[lb]{\smash{{\SetFigFont{7}{8.4}{\rmdefault}{\mddefault}{\updefault}{\color[rgb]{0,0,0}$w_1$}%
}}}}
\put(9751,-10036){\makebox(0,0)[lb]{\smash{{\SetFigFont{7}{8.4}{\rmdefault}{\mddefault}{\updefault}{\color[rgb]{0,0,0}$y_{k-s}$}%
}}}}
\put(9751,-13636){\makebox(0,0)[lb]{\smash{{\SetFigFont{7}{8.4}{\rmdefault}{\mddefault}{\updefault}{\color[rgb]{0,0,0}$z_1$}%
}}}}
\put(9751,-14086){\makebox(0,0)[lb]{\smash{{\SetFigFont{7}{8.4}{\rmdefault}{\mddefault}{\updefault}{\color[rgb]{0,0,0}$z_2$}%
}}}}
\put(9751,-14536){\makebox(0,0)[lb]{\smash{{\SetFigFont{7}{8.4}{\rmdefault}{\mddefault}{\updefault}{\color[rgb]{0,0,0}$z_3$}%
}}}}
\put(9751,-14986){\makebox(0,0)[lb]{\smash{{\SetFigFont{7}{8.4}{\rmdefault}{\mddefault}{\updefault}{\color[rgb]{0,0,0}$z_4$}%
}}}}
\put(9751,-15436){\makebox(0,0)[lb]{\smash{{\SetFigFont{7}{8.4}{\rmdefault}{\mddefault}{\updefault}{\color[rgb]{0,0,0}$z_{k-s}$}%
}}}}
\put(10426,-11236){\makebox(0,0)[lb]{\smash{{\SetFigFont{10}{12.0}{\rmdefault}{\mddefault}{\updefault}{\color[rgb]{0,0,0}$e_1$}%
}}}}
\put(10501,-14386){\makebox(0,0)[lb]{\smash{{\SetFigFont{10}{12.0}{\rmdefault}{\mddefault}{\updefault}{\color[rgb]{0,0,0}$e_2$}%
}}}}
\put(7126,-13636){\makebox(0,0)[lb]{\smash{{\SetFigFont{10}{12.0}{\rmdefault}{\mddefault}{\updefault}{\color[rgb]{0,0,0}$f$}%
}}}}
\put(8259,-12966){\makebox(0,0)[lb]{\smash{{\SetFigFont{8}{9.6}{\rmdefault}{\mddefault}{\updefault}{\color[rgb]{0,0,0}$v\ \!v$}%
}}}}
\put(6076,-12961){\makebox(0,0)[lb]{\smash{{\SetFigFont{7}{8.4}{\rmdefault}{\mddefault}{\updefault}{\color[rgb]{0,0,0}$x_5$}%
}}}}
\put(6526,-12961){\makebox(0,0)[lb]{\smash{{\SetFigFont{7}{8.4}{\rmdefault}{\mddefault}{\updefault}{\color[rgb]{0,0,0}$x_4$}%
}}}}
\put(6976,-12961){\makebox(0,0)[lb]{\smash{{\SetFigFont{7}{8.4}{\rmdefault}{\mddefault}{\updefault}{\color[rgb]{0,0,0}$x_3$}%
}}}}
\put(7426,-12961){\makebox(0,0)[lb]{\smash{{\SetFigFont{7}{8.4}{\rmdefault}{\mddefault}{\updefault}{\color[rgb]{0,0,0}$x_2$}%
}}}}
\put(7876,-12961){\makebox(0,0)[lb]{\smash{{\SetFigFont{7}{8.4}{\rmdefault}{\mddefault}{\updefault}{\color[rgb]{0,0,0}$x_1$}%
}}}}
\put(5502,-12957){\makebox(0,0)[lb]{\smash{{\SetFigFont{7}{8.4}{\rmdefault}{\mddefault}{\updefault}{\color[rgb]{0,0,0}$x_{k-2}$}%
}}}}
\end{picture}%
}%
        \qquad
        \subfigure[]
{%
\label{fig:b}
\begin{picture}(0,0)%
\includegraphics{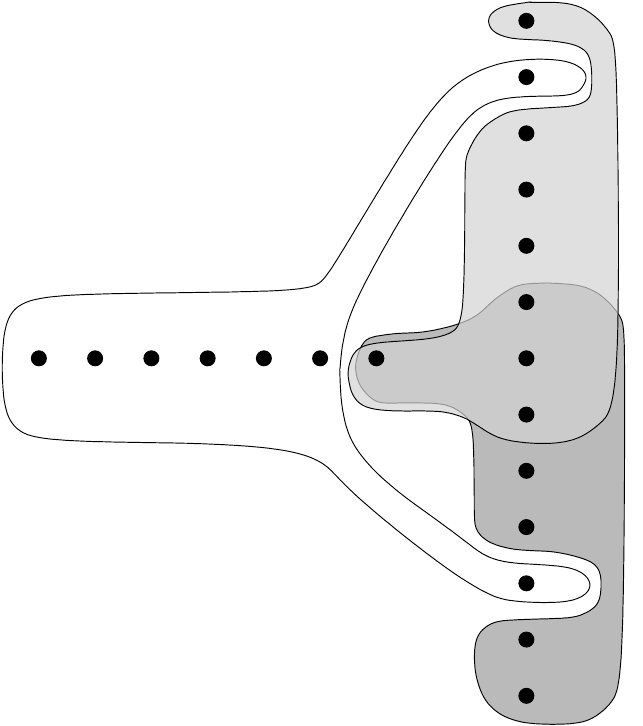}%
\end{picture}%
\setlength{\unitlength}{2368sp}%
\begin{picture}(5127,5804)(5389,-15601)
\put(9751,-12736){\makebox(0,0)[lb]{\smash{{\SetFigFont{7}{8.4}{\rmdefault}{\mddefault}{\updefault}{\color[rgb]{0,0,0}$w_2$}%
}}}}
\put(9751,-11836){\makebox(0,0)[lb]{\smash{{\SetFigFont{7}{8.4}{\rmdefault}{\mddefault}{\updefault}{\color[rgb]{0,0,0}$y_1$}%
}}}}
\put(9751,-11386){\makebox(0,0)[lb]{\smash{{\SetFigFont{7}{8.4}{\rmdefault}{\mddefault}{\updefault}{\color[rgb]{0,0,0}$y_2$}%
}}}}
\put(9751,-12286){\makebox(0,0)[lb]{\smash{{\SetFigFont{7}{8.4}{\rmdefault}{\mddefault}{\updefault}{\color[rgb]{0,0,0}$w_s$}%
}}}}
\put(9751,-10936){\makebox(0,0)[lb]{\smash{{\SetFigFont{7}{8.4}{\rmdefault}{\mddefault}{\updefault}{\color[rgb]{0,0,0}$y_3$}%
}}}}
\put(9751,-13186){\makebox(0,0)[lb]{\smash{{\SetFigFont{7}{8.4}{\rmdefault}{\mddefault}{\updefault}{\color[rgb]{0,0,0}$w_1$}%
}}}}
\put(9751,-10036){\makebox(0,0)[lb]{\smash{{\SetFigFont{7}{8.4}{\rmdefault}{\mddefault}{\updefault}{\color[rgb]{0,0,0}$y_{k-s}$}%
}}}}
\put(9751,-13636){\makebox(0,0)[lb]{\smash{{\SetFigFont{7}{8.4}{\rmdefault}{\mddefault}{\updefault}{\color[rgb]{0,0,0}$z_1$}%
}}}}
\put(9751,-14086){\makebox(0,0)[lb]{\smash{{\SetFigFont{7}{8.4}{\rmdefault}{\mddefault}{\updefault}{\color[rgb]{0,0,0}$z_2$}%
}}}}
\put(9751,-14986){\makebox(0,0)[lb]{\smash{{\SetFigFont{7}{8.4}{\rmdefault}{\mddefault}{\updefault}{\color[rgb]{0,0,0}$z_4$}%
}}}}
\put(9751,-15436){\makebox(0,0)[lb]{\smash{{\SetFigFont{7}{8.4}{\rmdefault}{\mddefault}{\updefault}{\color[rgb]{0,0,0}$z_{k-s}$}%
}}}}
\put(9751,-10486){\makebox(0,0)[lb]{\smash{{\SetFigFont{8}{9.6}{\rmdefault}{\mddefault}{\updefault}{\color[rgb]{0,0,0}$y_*$}%
}}}}
\put(9751,-14536){\makebox(0,0)[lb]{\smash{{\SetFigFont{8}{9.6}{\rmdefault}{\mddefault}{\updefault}{\color[rgb]{0,0,0}$z_*$}%
}}}}
\put(8361,-12961){\makebox(0,0)[lb]{\smash{{\SetFigFont{8}{9.6}{\rmdefault}{\mddefault}{\updefault}{\color[rgb]{0,0,0}$v$}%
}}}}
\put(10501,-14386){\makebox(0,0)[lb]{\smash{{\SetFigFont{10}{12.0}{\rmdefault}{\mddefault}{\updefault}{\color[rgb]{0,0,0}$e_2^{'}$}%
}}}}
\put(10426,-11236){\makebox(0,0)[lb]{\smash{{\SetFigFont{10}{12.0}{\rmdefault}{\mddefault}{\updefault}{\color[rgb]{0,0,0}$e_1^{'}$}%
}}}}
\put(7122,-13658){\makebox(0,0)[lb]{\smash{{\SetFigFont{10}{12.0}{\rmdefault}{\mddefault}{\updefault}{\color[rgb]{0,0,0}$e_3^{'}$}%
}}}}
\put(6076,-12961){\makebox(0,0)[lb]{\smash{{\SetFigFont{7}{8.4}{\rmdefault}{\mddefault}{\updefault}{\color[rgb]{0,0,0}$x_5$}%
}}}}
\put(6526,-12961){\makebox(0,0)[lb]{\smash{{\SetFigFont{7}{8.4}{\rmdefault}{\mddefault}{\updefault}{\color[rgb]{0,0,0}$x_4$}%
}}}}
\put(6976,-12961){\makebox(0,0)[lb]{\smash{{\SetFigFont{7}{8.4}{\rmdefault}{\mddefault}{\updefault}{\color[rgb]{0,0,0}$x_3$}%
}}}}
\put(7426,-12961){\makebox(0,0)[lb]{\smash{{\SetFigFont{7}{8.4}{\rmdefault}{\mddefault}{\updefault}{\color[rgb]{0,0,0}$x_2$}%
}}}}
\put(7876,-12961){\makebox(0,0)[lb]{\smash{{\SetFigFont{7}{8.4}{\rmdefault}{\mddefault}{\updefault}{\color[rgb]{0,0,0}$x_1$}%
}}}}
\put(5520,-12961){\makebox(0,0)[lb]{\smash{{\SetFigFont{7}{8.4}{\rmdefault}{\mddefault}{\updefault}{\color[rgb]{0,0,0}$x_{k-2}$}%
}}}}
\end{picture}%
}%
        \caption{Switching (a) before and (b) after.}
        \label{fig:1}
\end{figure}
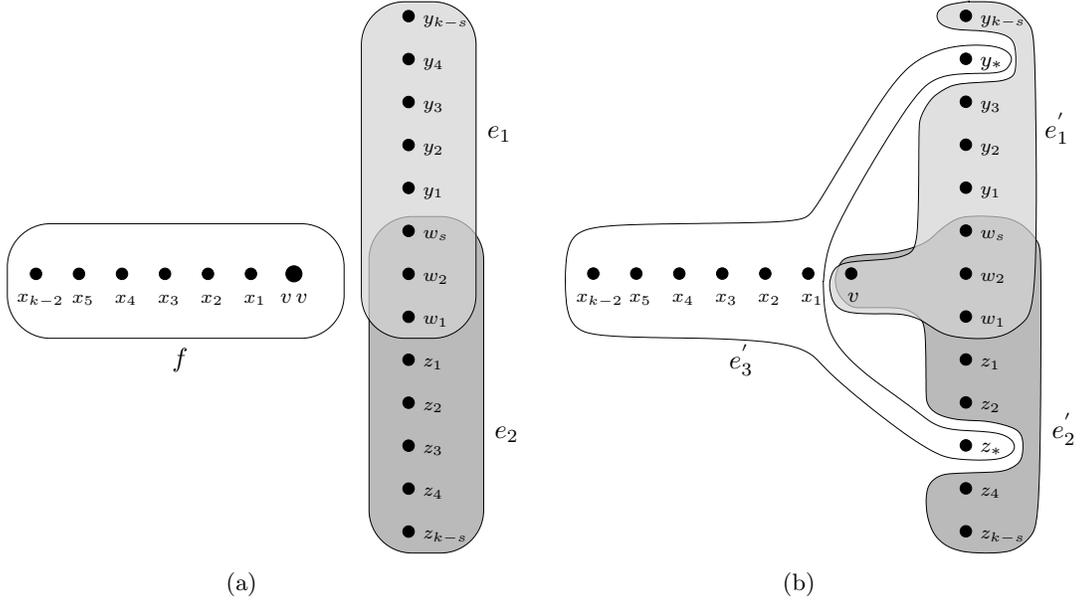

\subsection{Switchings}
Now we define an operation, called \emph{switching}, which generalizes  to $k$-graphs a graph switching introduced in \cite{MW} (see also \cite{MW2}).
    Permutations $\y \in \ec_l$, $\z \in \ec_{l-1}$ are said to be \emph{switchable}, if $\z$ can be obtained from $\y$ by the following operation. From the edges of $\y$, choose a loop~$f$ and two proper edges $e_1, e_2$ that are disjoint from $f$ and share at most $k-2$ vertices (see Figure~\ref{fig:a}). Letting $s = |e_1 \cap e_2|$, write
    \begin{equation*}\label{eq:fee}
    f = vvx_1\dots x_{k-2}, \qquad e_1 = w_1\dots w_s y_1\dots y_{k-s}, \qquad e_2 = w_1\dots w_s z_1 \dots z_{k-s}.
  \end{equation*}
  Select vertices $y_* \in \left\{ y_1, \dots, y_{k-s} \right\}$ and $z_* \in \{z_1, \dots, z_{k-s}\}$, and replace $f, e_1$, and $e_2$ by three proper edges
  \begin{equation*}
    e'_1 = e_1 \cup \{v\} - \{y_*\}, \qquad e'_2 = e_2 \cup \{v\} - \{z_*\}, \qquad e'_3 = f \cup \{y_*,z_*\}- \{v,v\}
  \end{equation*}
as in Figure~\ref{fig:b}.  
Since we are dealing with permutations, for definiteness let us assume that the procedure is performed by swapping with $y_*$ the copy of $v$ which appears in $\y$ further to the left and with $z_*$ the one further to the right.

  We can reconstruct permutations in $\ec_{l+1}$ which are switchable with $\y$ as follows. Pick a vertex $v \in [n]$, two edges $e_1'$, $e_2'$ containing $v$, and one more edge $e_3'$ (consult with Figure~\ref{fig:1} again). Choose a pair $\{y_*, z_*\}$  of vertices from $e_3'$; replace $e_i'$, $i=1,2,3$, by a loop and two edges defined as
 \begin{equation*}
   f = e_3' \cup \left\{ v, v \right\} \setminus \left\{ y_*, z_* \right\}, \qquad e_1 = e_1' \cup \left\{ y_* \right\} \setminus \left\{ v \right\}, \qquad  e_2 = e_2' \cup \left\{ z_* \right\} \setminus \left\{ v \right\}.
 \end{equation*}
\noindent Given $\y \in \ec_l$, let $F(\y)$ and $B(\y)$ stand, respectively, for the number of ways to perform the forward and backward switching,  or, in other words, the number of
  permutations $\x \in \ec_{l-1}$ and $\z \in \ec_{l+1}$ which are switchable with $\y$. Recall that $L=\sqrt{nd}$ and set $F_l=d^2n^2l$, $l=1,\dots,L$, and $ B =\tfrac{k-1}{2} n^2d^2 (d-1)$.
  \begin{proposition}
    \label{prop:regu}
   There is a sequence $\delta = \delta(n) = O((L+d^2)/dn)$  such that for all $\y \in \ec_l$, $0 < l \le L$
  \begin{equation*}
    (1-\delta) F_l\le F(\y) \le F_l \quad\mbox{ and }\quad  (1-\delta) B \le B(\y) \le B .
  \end{equation*}
  \end{proposition}
\begin{proof}
 Clearly 
      $F(\y) \le lm^2k^2 = n^2d^2l.$
 We say that two edges $e', e''$ of a $k$-graph are \emph{distant} from each other if their distance in the intersection graph of $H(\y)$ is at least three.
  Note that given $f, e_1,$ and $e_2$, some choice of $y_*$ and $z_*$ might not yield a permutation $\z \in \ec_{l-1}$, because one or more of $e_i'$'s might already be present in $\y$.  However, all $k^2$ choices of $(y_*, z_*)$ are allowed, if $e_1 \cap e_2 = \emptyset$ and both $e_1$ and $e_2$ are distant from $f$.
 Therefore, \begin{equation*}
   F(\y) \ge k^2 (m - l - 2k^2d^2)^2l = k^2m^2l(1 - O((L+d^2)/m)).
 \end{equation*}
 Clearly $B(\y) \le n(d)_2 m \binom k 2 = B$. To bound $B(\y)$ from below, we estimate the number of choices of $(v, e_1', e_2', e_3')$, for which at least one pair $\{y_*, z_*\}$ does not yield a permutation in $\ec_{l+1}$. This can only happen when one of $e_1', e_2', e_3'$ is a loop, which occurs for at most $2kldm + ln(d)_2$ choices, or when $e_3'$ is not distant from both $e_1'$ and $e_2'$, which occurs for at most $n(d)_2 \cdot 2k^2d^2$ choices. We have $B = \Theta(n^2d^3)$, therefore
\begin{equation*}
  B(\y) \ge B - \binom k 2 \left( 2kldm + ln(d)_2 + 2k^2nd^4 \right) = B\left(1 - O\left( \frac{L + d^2}{nd} \right)\right).
 \end{equation*}
\end{proof}

\begin{proof}[Proof of Theorem \ref{thm:enum}]
  Counting the switchable pairs $\y \in \ec_l$, $\z \in \ec_{l-1}$ in two ways, from Proposition \ref{prop:regu} we conclude that
  \begin{equation}\label{eq:ratio}
    \frac{(1-\delta)  B }{ F_l} \le \frac{|\ec_l|}{|\ec_{l-1}|} \le  \frac{ B }{(1-\delta) F_l}.
  \end{equation}
  Since $B / F_l=(k-1)(d-1)/2l$, from \eqref{eq:sumprod} and \eqref{eq:ratio} we get 
  \begin{equation*}
    \sum_{l = 0}^L \frac{x^l}{l!} \le \frac{|\ec|}{|\ec_0|} \le \sum_{l = 0}^L \frac{y^l}{l!}
  \end{equation*}
  where $x = \frac12(1-\delta)(k-1)(d-1)$ and $y = \frac12(k-1)(d-1)/(1-\delta)$.
  Therefore by Taylor's theorem $|\ec|/|\ec_0|$ is at most $\e^y$ and at least
  \begin{equation*}
    \e^x(1-x^L/L!) \ge \e^x (1 - (\e x/L)^L) = \exp \left\{ x - o \left( \sqrt{d/n} \right) \right\},
  \end{equation*}
  the inequality following from a standard fact $L! \ge (L/\e)^L$.
  Since $x, y = (k-1)(d-1)/2 + O(\sqrt{d/n} + d^2/n)$, we get 
  $$\frac{|\ec|}{|\ec_0|} =\exp\left\{\frac12(k-1)(d-1)+ O(\sqrt{d/n} + d^2/n)\right\}$$
  which together with (\ref{eq:pcec}) implies (\ref{suff}), hereby completing the proof.
\end{proof}

\section{Concluding remarks}\label{cr}

\begin{remark}  We believe that for $k=3$ the constraint $d=o(n^{1/2})$ in Theorem \ref{thm:enum} can be relaxed to $d=o(n)$ by allowing $O(d^2/n)$ multiple edges in $\y\in\ec$ and applying an appropriate  switching technique to eliminate them along with the loops.
\end{remark}

\begin{remark}
In  a forthcoming paper \cite{loose} we apply the switching technique presented here to embed asymptotically almost surely (\emph{a.a.s.}) an ordinary Erd\H os-R\'enyi random $k$-graph 
$\hh^{(k)}(n,m')$, $k\ge3$, into a random $d$-regular $k$-graph $\hh^{(k)}(n,d)$ for $d=\Omega(\log n)$,  $d=o(\sqrt n)$ and $m'=cnd/k$, for some constant $c>0$. Consequently, \emph{a.a.s.} $\hh^{(k)}(n,d)$ inherits from $\hh^{(k)}(n,m')$ all increasing properties held by the latter model.
\end{remark}

\begin{remark}\label{alg} An algorithm of McKay and Wormald~\cite{MW} can be easily adapted to $k$-graphs, yielding an expected polynomial time uniform generation of $d$-regular $k$-graphs in $\hcnd$.
The algorithm keeps selecting  a random permutation $\y\in \pc$ until $\y\in \ec$. Then, iteratively, a random switching  is applied $\lambda(\y)$  times to eliminate all loops and finally yield a random element of $\ec_0$. This  leads to an \emph{almost} uniform distribution over $\hcnd$. To make it \emph{exactly} uniform, McKay and Wormald applied an ingenious trick of restarting the whole algorithm after every iteration of switching, say from $\y\in\ec_l$ to $\z\in\ec_{l-1}$, with probability $1-( F(\y)(1-\delta_1) B )/( B(\z) F_l)\le2\delta_1$.
However, the assumption on $d$ has to be strengthened, so that the reciprocal of the probability of not restarting the algorithm before its successful termination, or $(1-\phi_k(n))^{-1}(1-2\delta_1(n))^{-L}=e^{O(\delta_1(n)L)}$, is at most a polynomial function of $n$. With our choice of $L$ this imposes the bound $d=O(n^{1/3}(\log n)^{2/3})$. We may push it up to $d=O(\sqrt{n\log n})$ by redefining $L=kd+\omega(n)$ for any (sufficiently slow) sequence $\omega(n)\to\infty$. This change requires that in the last part of the proof of Proposition~\ref{prop:YE}, instead of the first moment, Chebyshev's inequality is used (see Appendix~\ref{appendix:a}).

\end{remark}



\newpage

\appendix
\section{Appendix}\label{appendix:a}
\begin{proof}[ Proof of Proposition \ref{prop:YE}]

  We will show that each of the following four statements holds with probability $1-O(\sqrt{d/n} + d^2/n^{k-2})$:
    \begin{itemize}
      \item[(i)] $\Y$ has no multiple edges,
      \item[(ii)] $\Y$ has no edge with a vertex of multiplicity at least 3,
      \item[(iii)] $\Y$ has no edge with two vertices of multiplicity at least 2,
      \item[(iv)] $\lambda(\Y) \le L$.
    \end{itemize}
    
\noindent    
(i) The probability that two particular edges of $\Y$ are identical as multisets equals
$$
\sum_{k_1+\cdots+k_n=k}\binom{k}{k_1,\dots,k_n}^2\frac{\binom{dn-2k}{d-2k_1,\dots,d-2k_n}}{\binom{dn}{d,\dots,d}}\le k!^2\sum\frac{d^{2k}}{(dn)_{2k}} = O\left(n^k\frac{d^{2k}}{(dn)_{2k}}\right)=O(n^{-k}),
$$
therefore, by our assumption on $d$, the expected number of pairs of multiple edges does not exceed
\begin{equation*}
  O\left(\binom{m} 2 n^{-k}\right) = O(d^2n^{2-k}).
\end{equation*}

\noindent    
(ii) The expected number of edges of $\Y$ having a vertex of multiplicity at least $3$ is at most
$$m\times \binom k3\times n\times\frac {\binom{dn-3}{d-3,d,\dots,d}}{\binom{dn}{d,\dots,d}}= m\binom k3 n\frac{(d)_3}{(dn)_3}=O(d/n).$$

\noindent    
(iii) Similarly, the expected number of edges of $\Y$ having at least two vertices of multiplicity at least $2$ is at most
$$m\times k^4\times n^2\times\frac {\binom{dn-4}{d-2,d-2,d,\dots,d}}{\binom{dn}{d,\dots,d}}= m k^4 n^2\frac{(d)_2^2}{(dn)_4}=O(d/n).$$

\noindent    
(iv) In view of (ii) and (iii), it is enough to show that the number of loops of the form $x_1x_1x_2x_3\dots x_{k-1}$ does not exceed $L$. For $i = 1, \dots, m$, let $\I_i$ be the indicator of the event that the $i$'th edge of $\Y$ is such a loop. Hence, $\lambda(\Y) = \sum_{i = 1}^m \I_i$. For every $i$ we have
  \begin{equation*}
    \E \I_i = \frac{\binom k 2 (n)_{k-1}(d)_2d^{k-2}}{(nd)_k} \sim \binom k 2 \frac{d-1}dn^{-1}.
    \label{eq:EI}
  \end{equation*}
  Therefore
  \begin{equation}\label{eq:EY}
    \E\lambda(\Y) \sim \frac {k-1} 2(d-1),
  \end{equation}
  and by Markov's inequality,
  $$\Prob(\lambda(\Y)>L)\le\frac{\E\lambda(\Y)}L=O(d^{1/2}n^{-1/2}) $$
  \end{proof}

\begin{proof}[Proof that $\prob{\lambda(\Y)>kd+\omega(n)}=o(1)$]

Let $L:=kd+\omega(n)$.
We will show that $\Var \lambda(\Y)= O( d)$, from which the desired fact follows by \eqref{eq:EY} and Chebyshev's inequality:
  $$\prob{\lambda(\Y) > L} \le \frac{\Var\lambda(\Y)  }{(L-\E \lambda(\Y) )^2} =O\left(\frac{d}{(d+\omega(n))^2}\right)=O((d+\omega(n))^{-1})=o(1).$$
 Recall that $\I_i$ is the indicator  that the $i$'th edge of $\Y$ is a loop with only one repetition, $\lambda(\Y) = \sum_{i = 1}^m \I_i$, and for every $i$ we have $\E \I_i \sim \binom k 2 \frac{d-1}dn^{-1}.$
  If $i \neq j$, then
  $$\E \I_i \I_j \le \frac{\binom k 2^2(n)_{k-1}^2(d)_2^2d^{2k-4}}{(nd)_{2k}},$$
  therefore
  \begin{multline*}
    \Cov (\I_i, \I_j) = \E \I_i\I_j - \E \I_i \E \I_j \\
    \le \frac{\binom k 2^2(n)_{k-1}^2(d)_2^2d^{2k-4}}{(nd)_{2k}(nd)_k}\left( (nd)_k - (nd-k)_k \right) = O(n^{-3}d^{-1}).
  \end{multline*}
Finally we get
\begin{multline*}
  \Var \lambda(\Y)  = \sum_{1 \le i \le m} \Var \I_i + \sum_{1 \le i \neq j \le m} \Cov(\I_i,\I_j) 
  = O(mn^{-1} + m^2n^{-3}d^{-1} ) = O(d).
\end{multline*}
  \end{proof}

\end{document}